\documentclass[11pt]{article}%
\usepackage{amsmath}
\usepackage{amsfonts}
\usepackage{amssymb}
\usepackage{graphicx}%
\providecommand{\U}[1]{\protect\rule{.1in}{.1in}}
\newtheorem{theorem}{Theorem}

\newtheorem{lemma}[theorem]{Lemma}

\newenvironment{proof}[1][Proof]{\noindent\textbf{#1.} }{\ \rule{0.5em}{0.5em}}
\begin{document}

\title{Deforming the Window of a Gabor Frame: the Ellipsoid Method}
\author{Maurice A. de Gosson\\University of Vienna, NuHAG}
\maketitle

\begin{abstract}
In a recent paper in \textit{Appl. Comput. Harmon. Anal.} \textbf{38}(2),
196--221 (2014) we have introduced and studied the notion of weak Hamiltonian
deformation of a Gabor (=Weyl--Heisenberg) frame. In this Note we use these
results to prove that one can modify the window of a Gabor frame using certain
metaplectic operators provided that one modifies only a finite number of
points of the frame lattice.

\end{abstract}

\section{Introduction}

In a previous work \cite{go14} we initiated the study of weak Hamiltonian
deformations of Gabor frames using ideas borrowed from semiclassical
mechanics. In this Note we show that this method allows to modify the window
of a Gabor frame using families of metaplectic operators while only replacing
a finite number of lattice points with new points contained in an arbitrarily
chosen phase space (or time-frequency) ellipsoid. As in \cite{go14} we define
Gabor frames as follows: let $\phi$ be a non-zero square integrable function
on ${\mathbb{R}}^{n}$ (hereafter called the \textit{window}) and $\Lambda$ a
discrete subset of ${\mathbb{R}}^{2n}$ (the \textit{lattice}). The
corresponding \textit{Gabor, or Weyl--Heisenberg, }system is the set%
\begin{equation}
{\mathcal{G}}^{\hbar}(\phi,\Lambda)=\{\widehat{T}(z)\phi:z\in\Lambda
\}\label{frame}%
\end{equation}
where $\widehat{T}(z)=e^{-i\sigma(\hat{z},z)/\hbar}$ is the Heisenberg
operator. ${\mathcal{G}}^{\hbar}(\phi,\Lambda)$ is called a $\hbar
$-\textit{frame} for $L^{2}({\mathbb{R}}^{n})$ if there exist constants
$A,B>0$ (the \textit{frame bounds}) such that
\begin{equation}
A||\psi||^{2}\leq\sum_{z\in\Lambda}|(\psi|\widehat{T}(z)\phi)|^{2}\leq
B||\psi||^{2}\label{frame1}%
\end{equation}
for every $\psi\in L^{2}({\mathbb{R}}^{n})$. When the parameter $\hbar$ is
chosen equal to $1/2\pi$ the notion of $\hbar$-frame coincides with the usual
\textquotedblleft naive\textquotedblright\ notion of Gabor frame \cite{gr01}.
The advantage of using definition (\ref{frame1}) is that it highlights the
symplectic covariance of Gabor frames: if $S\in\operatorname*{Sp}(n)$ is a
symplectic automorphism of ${\mathbb{R}}^{2n}$ then ${\mathcal{G}}^{\hbar
}(\widehat{S}\phi,S\Lambda)$ is a $\hbar$-frame if and only it is the case for
${\mathcal{G}}^{\hbar}(\phi,\Lambda)$ ($\widehat{S}$ is anyone of the two
metaplectic operators corresponding to $S$).

In this Note we prove a simple consequence of the weak Hamiltonian deformation
theory; we call it the \textquotedblleft ellipsoid method\textquotedblright%
\ because it relies on the fact that a phase space ellipsoid may be viewed as
the energy hypersurface of a Hamilton function that is a quadratic form in the
$x_{j},p_{k}$ variables. We will see in a forthcoming work that this method
can be extended to arbitrary Hamiltonian functions.

\section{Weak Hamiltonian Deformations}

Let $H\in C^{2}(\mathbb{R}^{2n})$; we denote by $(f_{t}^{H})$ the Hamiltonian
function $H$. In \cite{go14} we proved the following covariance result for
$\hbar$-frames:

\begin{theorem}
\label{thm1}Let $\Lambda_{t}=f_{t}^{H}(\Lambda)$, $z_{t}=f_{t}^{H}(z_{0})$,
$\widehat{S}_{t}(z_{0})=Df_{t}^{H}(z_{0})$ ($z_{0}$ an arbitrary point in
$\mathbb{R}^{2n}$). For $\phi\in L^{2}(\mathbb{R}^{n})$ set
\begin{equation}
\phi_{t}=\widehat{T}(z_{t})\widehat{S}_{t}(z_{0})\widehat{T}(z_{0})^{-1}%
\phi.\label{fit}%
\end{equation}
The Gabor system ${\mathcal{G}}^{\hbar}(\phi_{t},\Lambda_{t})$ is a $\hbar
$-frame if and only if ${\mathcal{G}}^{\hbar}(\phi,\Lambda)$ is a $\hbar
$-frame; when this the case both frames have the same bounds.
\end{theorem}

The proof of this result is bases on the commutation and addition properties
of the Heisenberg operators, and on their symplectic covariance
\cite{go06,go11}. Formula (\ref{fit}) corresponds (up to an unimportant phase
factor) to the semiclassical propagation \cite{he75,li86} of a wavepacket
centered at the point $z_{0}$.

Theorem \ref{thm1} has been recently applied \cite{gogrro15} to the study of
the stability of Gabor frames under small time Hamiltonian evolutions.

\section{Main Result: Precise Statement and Proof}

We assume from now on that the lattice $\Lambda$ is \textquotedblleft$\delta
$-separated\textquotedblright: there exists $\delta>0$ such that
$|z-z^{\prime}|>0$ for all $(z,z^{\prime})\in\Lambda\times\Lambda$ such that
$z\neq z^{\prime}$. Under this assumption we have:

\begin{lemma}
\label{Lemma1}Let $\Sigma$ be a closed hypersurface in phase space
$\mathbb{R}^{2n}$ and let
\[
\Sigma_{\varepsilon}=%
{\textstyle\bigcup\limits_{z\in\Sigma}}
\overline{B}^{2n}(z,\varepsilon)
\]
be the \textquotedblleft$\varepsilon$-thickening\textquotedblright\ of
$\Sigma$ ($\overline{B}^{2n}(z,\varepsilon)$ is the closed ball with radius
$\varepsilon$ centered at $z$). There exists $\varepsilon$ such that
$\Sigma_{\varepsilon}\cap\Lambda=\Sigma\cap\Lambda$. 
\end{lemma}

This Lemma means that if \ is chosen small enough, the $\varepsilon
$-thickening\ of $\Sigma$ will contain no more points of the lattice than
those which are already on the hypersurface $\Sigma$. 

We note that surprisingly enough, even in the simple case where $\Sigma$ is an
ellipsoid as will be supposed below, the question of determining how many
points of the lattice belong to $\Sigma$ or its interior is in general very
difficult; it harks back to an early result of Landau \cite{la15,la24} and has
unexpected connections with number theory. We recall that a closed
hypersurface is orientable, and hence separates $\mathbb{R}^{2n}$ in two
connected components, one of them (\textquotedblleft the
interior\textquotedblright) being bounded (Jordan--Brouwer theorem).

Let us view the quadratic function $H(z)=\frac{1}{2}Mz\cdot z$ ($M$ positive
definite) as a Hamiltonian function. The corresponding Hamilton equations of
motion are \cite{ar89,go06,go11} $\dot{z}(t)=JMz(t)$ hence their solution is
given by the simple formula $z(t)=e^{tJM}z(0)$. As is customary, we call the
mapping $S_{t}$ taking the initial value $z(0)$ to the solution $z(t)$ at time
$t$ the \textquotedblleft Hamiltonian flow\textquotedblright\ determined by
$H$. We thus have $z(t)=S_{t}z(0)$ and we may identify $S_{t}$ with the matrix
$e^{tJM}$, which is symplectic. We thus have $S_{t}\in\operatorname*{Sp}(n)$
for every $t\in\mathbb{R}$. It now follows from general principles (the path
lifting theorem, see \cite{go06} and the references therein) that when $t$
varies, $S_{t}$ describes a curve in $\operatorname*{Sp}(n)$ passing through
the identity at time $t=0$. This curve can be lifted to a unique curve
$t\longmapsto\widehat{S}_{t}$ in $\operatorname*{Mp}(n)$ passing through the
identity at time $t=0$. Here $\operatorname*{Mp}(n)$ is the metaplectic group
(it is a unitary representation of the double covering group of
$\operatorname*{Sp}(n)$; for a detailed study of $\operatorname*{Mp}(n)$ see
\cite{fo89,go06,go11}).

\begin{theorem}
\label{thm2}Let ${\mathcal{G}}^{\hbar}(\phi,\Lambda)$ be a Gabor system.
Consider the ellipsoid
\[
\Sigma=\{z\in{\mathbb{R}}^{n}:\tfrac{1}{2}Mz\cdot z=E\}
\]
and set $F=\Omega\cap\Lambda=\emptyset$ ($\Omega$ the compact set bounded by
$\Sigma$). Let $S_{t}=e^{tJM}\in\operatorname*{Sp}(n)$ and let $\widehat
{S}_{t}\in\operatorname*{Mp}(n)$ be obtained as described above. Set
$F_{t}=S_{t}(F)$. Then ${\mathcal{G}}^{\hbar}(\widehat{S}_{t}\phi
,\Lambda^{\prime}\cup F_{t})$ is a $\hbar$-frame if and only if ${\mathcal{G}%
}^{\hbar}(\phi,\Lambda)$ is $\hbar$-frame (with same frame bounds). In
particular, if $F=\emptyset$ then hen ${\mathcal{G}}^{\hbar}(\widehat{S}%
_{t}\phi,\Lambda)$ is a $\hbar$-frame if and only if ${\mathcal{G}}^{\hbar
}(\phi,\Lambda)$ is a $\hbar$-frame.
\end{theorem}

\begin{proof}
That $S_{t}\in\operatorname*{Sp}(n)$ is clear since $JM\in\mathfrak{sp}(n)$
(the symplectic Lie algebra). Let $\Omega$ be the compact set bounded by
$\Sigma$. Choosing $\varepsilon$ such that $\Sigma_{\varepsilon}\cap
\Lambda=\Sigma\cap\Lambda$ as in Lemma \ref{Lemma1}, let $\chi_{\varepsilon}$
be  a smooth function on ${\mathbb{R}}^{2n}$ such that
\begin{equation}
\chi_{\varepsilon}(z)=\left\{
\begin{array}
[c]{c}%
1\text{ \ \textit{if} \ }z\in\Omega\cup\Sigma_{\varepsilon/2}\\
0\text{ \ \textit{if} \ }z\notin\Omega\cup\Sigma_{\varepsilon}%
\end{array}
\right.  \label{kieps}%
\end{equation}
(the existence of $\chi_{\varepsilon}$ follows from Tietze--Urysohn's lemma).
Setting $H^{\varepsilon}=H\chi_{\varepsilon}$ the support of $H^{\varepsilon}$
is $\Omega\cup\Sigma_{\varepsilon}$ and we have $H^{\varepsilon}(z)=H(z)$ for
$z\in\Omega\cup\Sigma_{\varepsilon/2}$. The Hamiltonian flow $(f_{t}%
^{\varepsilon})$ determined by $H^{\varepsilon}$ thus satisfies
\begin{equation}
f_{t}^{\varepsilon}(z)=\left\{
\begin{array}
[c]{c}%
S_{t}z\text{ \ \textit{if} \ }z\in\Omega\cup\Sigma_{\varepsilon/2}\\
z\text{ \ \textit{if} \ }z\notin\Omega\cup\Sigma_{\varepsilon}%
\end{array}
\right.  .\label{feps}%
\end{equation}
Notice that if $z\in\Sigma$ then $S_{t}z\in\Sigma$ for all $t\in\mathbb{R}$
since $\Sigma$ is an \textquotedblleft energy hypersurface\textquotedblright%
\ for $H$. The result now follows from Theorem \ref{thm1} taking $z_{0}=0$
since $z_{t}=0$ and $S_{t}(0)=S_{t}$.
\end{proof}

It is easy to prove a converse to Theorem \ref{thm2}: given a symplectic
matrix $S=e^{X}$ ($X\in\mathfrak{sp}(n)$, the symplectic Lie algebra) we can
associate a symplectic flow $S_{t}=e^{tX}$, namely that of the ellipsoid
$\Sigma$ with $JM=X$. The case of an arbitrary $S\in\operatorname*{Sp}(n)$ is
slightly more complicated since it leads to time-dependent Hamiltonians: for
an arbitrary symplectic matrix $S$ it is not true in general that $S=e^{X}$
for some $X\in\mathfrak{sp}(n)$; however one can show \cite{HZ,po01} that
there exists a $C^{1}$ path $t\longmapsto S_{t}$ of symplectic matrices such
that $S_{0}=I$ and $S_{1}=S$. To this (not uniquely defined) path corresponds
a generally time-dependent quadratic Hamiltonian $H$ such that $f_{t}%
^{H}=S_{t}$. This case will be treated in detail in a forthcoming publication.

\section{Discussion}

Theorem \ref{thm2} is a straightforward consequence of Theorem \ref{thm1}
obtained by truncating the initial Hamiltonian: the generalization of Theorem
\ref{thm2} to arbitrary Hamiltonian functions can be made along the same
lines. All these results are easily extended (under certain supplementary
assumptions) to the case where the window $\phi$ belongs to a modulation space
(cf. \cite{go14,gogrro15}).

\end{document}